\documentclass[11pt, twoside, a4paper, english, reqno]{amsart}
\usepackage{amssymb,amsfonts,latexsym,color}
\usepackage{graphics,verbatim}
\usepackage{graphicx}
\usepackage{hyperref}
\newtheorem{theorem}{Theorem}[section]
\newtheorem{proposition}{Proposition}[section]
\newtheorem{lemma}{Lemma}[section]
\newtheorem{corollary}{Corollary}[section]
\newtheorem{remark}{Remark}[section]

\newcommand{\<}{\left\langle}
\renewcommand{\>}{\right\rangle}

\newcommand{\eps}{\varepsilon}

\newcommand{\be} {\begin{equation}}
\newcommand{\ee} {\end{equation}}
\newcommand{\bea} {\begin{eqnarray}}
\newcommand{\eea} {\end{eqnarray}}
\newcommand{\Bea} {\begin{eqnarray*}}
	\newcommand{\Eea} {\end{eqnarray*}}

\newcommand{\al} {\alpha}
\newcommand{\ba} {\beta}
\newcommand{\de} {\delta}

\newcommand{\Om} {\Omega}

\newcommand{\De} {\Delta}
\newcommand{\la} {\lambda}

\newcommand{\lab} {\label}

\newcommand{\f}{\frac}

\newcommand{\R}{\mathbb R}
\newcommand{\N}{\mathbb N}
\newcommand{\Rn}{\mathbb R^N}
\newcommand{\Iom}{\int_{\Omega}}
\newcommand{\deb}{\rightharpoonup}

\catcode`\@=11

\makeatletter \@addtoreset{equation}{section} \makeatother

\begin{document}
	
	\title[On the existence results for nonlocal system
	%$(p_1,p_2)$-fractional Laplacian system 
	with lack of compactness]{Non-homogeneous $(p_1,p_2)$-fractional Laplacian systems with lack of compactness}

\author{Debangana Mukherjee}
\address{Department of Mathematics, Indian Institute of Science Education and Research, Dr. Homi Bhaba Road, Pune, 411008, India}
\email{debangana18@gmail.com}
	\author{Tuhina Mukherjee}
\address{Department of Mathematics, Indian Institute of Technology Jodhpur, Rajasthan-506004, India-342037 }
\email{tulimukh@gmail.com, tuhina@iitj.ac.in}
	
	\date{}

 \begin{abstract}
The present paper studies the existence of weak solutions for following type of non-homogeneous system of equations
\begin{equation*}
(S)
\left\{\begin{aligned}
(-\Delta)^{s_1}_{p_1} u  &=u|u|^{\al-1}|v|^{\ba+1}+f_1(x)  \,\mbox{ in }\, \Om, \\
(-\Delta)^{s_2}_{p_2} v  &=|u|^{\al+1}v|v|^{\ba-1}+f_2(x) \,\mbox{ in }\, \Om, \\
u=v &= 0  \,\mbox{ in }\, \Rn \setminus \Om, \\
\end{aligned}
\right.
\end{equation*}
where $\Om \subset \Rn$ is smooth bounded domain, $s_1,s_2 \in (0,1)$, $1<p_1,p_2<\infty$, $N>\max\{p_1s_1,p_2s_2\}$, $\alpha>-1$ and $\beta>-1$. We employ the variational techniques where the associated energy functional is minimized over Nehari manifold set while imposing appropriate bound on dual norms of $f_1,f_2$.
\end{abstract}

\maketitle

%\tableofcontents
\section{Introduction}
The study of elliptic equations and systems involving fractional and $p$-fractional Laplace operators have dynamized a lot of application in nonlinear analysis. In recent years, enormous attention is given to the research on very delicate reading about the existence and multiplicity of solutions for such elliptic problems. 
The present article deals with non-homogeneous system of equations involving $(p_1,p_2)$-type fractional operators $(-\De)^{s_i}_{p_i}, i=1,2$ defined as:
\begin{equation}\label{frac}
(-\Delta)_{p_i}^{s_i}u(x) := \lim_{\epsilon\to 0}\int_{\mathbb R^N\setminus B_\epsilon(x)} \frac{|u(x)-u(y)|^{p_i-2}(u(x)-u(y))}{|x-y|^{N+p_is_i}}~dy,\;\;\text{for}\;\; x \in \mathbb R^N
\end{equation}
where $u\in C_c^\infty(\R^N)$ and $B_\epsilon(x)$ denotes ball in $\R^N$ with centre $x$ and radius $\epsilon$. 
%The problem studied here is two fold-first over bounded domains and second over the unbounded domains. 
When $s_1=s_2$ the equation reduces to a $(p_1,p_2)$-type Laplacian problem which appears in a more general reaction-diffusion system 
\begin{equation}\label{intro-1}
    u_t =\text{div}(a(u)\nabla u)+ g(x,u)
\end{equation}
where $a(u)= |\nabla u|^{p-2}\nabla u+ |\nabla u|^{q-2}\nabla u$. Such problems have a wide range of applications in physics and related sciences such as biophysics, plasma physics, and chemical reaction
design, etc. where $u$ describes a concentration, and the first
term on the right-hand side of \eqref{intro-1} corresponds to a diffusion with a diffusion coefficient $a(u)$; the term
$g(x,u)$ stands for the reaction, related to sources and energy-loss processes. A lot of attention has been given to the study of $(p_1,p_2)$-Laplace equations in the last few years, for instance, we cite some attributing references
\cite{faria-miyagaki, deepak, MP, sidi, tanaka, yang-yin}.

We are attentive in considering the system of non-homogeneous equations involving $((s_1,p_1),(s_2,p_2))$ fractional Laplacian operators. In case $s_1=s_2=1$, the non-homogeneous elliptic systems were initially researched by J. Chabrowski \cite{chabro}  where the author has established multiple solutions for 
\begin{equation*}
\begin{aligned}
(-\Delta)_{p_1} u  &=\la u|u|^{\al-1}|v|^{\ba+1}+f_1  \,\mbox{ in }\, \Om, \\
(-\Delta)_{p_2} v  &=\la |u|^{\al+1}v|v|^{\ba-1}+f_2 \,\mbox{ in }\, \Om, \\
u=v &= 0  \,\mbox{ in }\, \R^N \setminus \Om, \\
\end{aligned}
\end{equation*}
where $\Om \subset \Rn$ is bounded, $\la>0$, $1<p_1,p_2<N$, $-1<\alpha,\beta$.  If we take $s_1=s_2=1$, then our problem $(S)$ is also studied by J. Velin \cite{velin-1} over a regular bounded set $\Om$ and by Benmouloud et al. \cite{BES} over $\Om$ any open set. The couple of papers have adopted variational techniques with enthralling use of a maximum norm over the product function space. Later on, a class of $(p,q)$-Laplacian system with indefinite weights have been learned by Afrouzi et al. in \cite{afrouzi} via local minimization techniques.
For interested readers, we point out some striking articles \cite{shivaji, HA, miyagaki, rasouli-2, rasouli-1} which deals with various kind of elliptic systems involving the $(p,q)$-Laplace operator. 
In \cite{shivaji}, the authors have analyzed the existence of positive solutions to the non-cooperative fractional elliptic system and the scheme heavily relies upon the Schauder fixed point theorem. Using the mechanism of sub-super solutions, the authors in \cite{HA} have achieved captivating results about the existence of a weak solution.

%\begin{equation*}
%\left\{\begin{aligned}
%(-\Delta)^{s_1}_{p_1} u  &=u|u|^{\al-1}|v|^{\ba+1}+f_1  \,\mbox{ in }\, \Om, \\
%(-\Delta)^{s_2}_{p_2} v  &=|u|^{\al+1}v|v|^{\ba-1}+f_2 \,\mbox{ in }\, \Om, \\
%u=v &= 0  \,\mbox{ in }\, \Rn \setminus \Om, \\
%\end{aligned}
%\right.
%\end{equation*}
%where $\Om \subset \mathbb R^N$ is a smooth bounded domain,  $s_1,s_2 \in (0,1)$, $1<p_1,p_2<\infty$, $N>\max\{p_1s_1,p_2s_2\}$, $\alpha>-1$ and $\beta>-1$.

Driven by the above research, our prospect in this article is to establish the existence results for $(S)$. There are very few and sparse works till now pledging with $((s_1,p_1), (s_2,p_2))$- fractional Laplacian system of equations, for instance, we cite \cite{deba-tuhi-1}. We also present a very recent and stirring article \cite{mousomi-1} which deals with non-homogeneous semilinear fractional system with lack of compactness. The essence of our article lies in the fact that we contemplate on systems of equation with fractional $(p_1,p_2)$-Laplace operator, particularly, the quasilinear case over a bounded domain which is first of its kind in literature. Moreover, we are not assuming any fixed relation between $s_1,s_2$. Since the nonlinearity is of critical growth, we face a lack of compactness situation even in the bounded domain case. {The approach to handle our problem is variational and we do a local minimization to establish existence results}. 
%We remark that this problem can be extended to considering  more general non local operator $\mathcal L_{\phi_i}$ instead of $(-\Delta)^{s_i}_{p_i}$.

\textit{Organisation of the paper}: The present paper is classified into the following sections- Section 2 which is again divided into four parts- first part deals with the suitable function spaces with which we have worked on,
second part intents to cover some notations and symbols used throughout the paper,  third part is schemed for some hypotheses needed to establish our main outcome and the last part is designed for our main result. Section 3 embraces study of the problem over bounded domains and the proof of our main result.

\section{Functional Setting and Main result}

\subsection{Functional Setting}
 In this section, we interpret appropriate function spaces which are imperative for our analysis. Let $ p> 1,\,s\in(0,1),\, N>ps,\, p_s^*:=\frac{Np}{N-sp}.$ We denote the standard fractional Sobolev space by $W^{s,p}(\Omega)$ endowed with the norm
$$
\|{u}\|_{W^{s,p}(\Om)}:=\|{u}\|_{L^p(\Om)}+\left(\int_{\Om\times\Om} \frac{|u(x)-u(y)|^p}{|x-y|^{N+sp}}dxdy\right)^{1/p}.
$$
We set $Q:=\R^{2N}\setminus (\Om^c \times \Om^c)$, where $\Om^c=\Rn \setminus \Om$ and define $$
X_{s,p}(\Om):=\Big\{u:\mathbb{R}^N\to\mathbb{R}\mbox{ measurable }\Big|u|_{\Omega}\in L^p(\Omega)\mbox{ and }
\int_{Q} \frac{|u(x)-u(y)|^p}{|x-y|^{N+sp}}dxdy<\infty\Big\}.
$$
The space $X_{s,p}(\Om)$ is endowed with the norm defined as
$$\|u\|_{s,p}:=|u|_{L^p(\Om)}+\left(\int_{Q} \frac{|u(x)-u(y)|^p}{|x-y|^{N+sp}}dxdy\right)^{1/p}.$$
%where $|u|_p:=\left(\int_\Om |u(x)|^p dx\right)^\frac{1}{p}$. 
We note that in general $W^{s,p}(\Om)$ is not same as $X_{s,p}(\Om)$ as $\Om\times\Om$ is strictly contained in $Q$.
We define the space $X_{0,s,p}(\Om)$ as
$$X_{0,s,p}(\Om) :=\Big\{u \in X_{s,p} : u=0 \quad\text{a.e. in}\quad \Rn \setminus \Om\Big\} $$
or equivalently
as $\overline{C_0^\infty(\Om)}^{X_{s,p}(\Om)}$. It is well-known that for $p>1$,  $X_{0,s,p}(\Om)$ is a uniformly convex Banach space endowed with the norm
$$[u]_{s,p}:=\|u\|_{0,s,p}=\left(\int_{Q} \frac{|u(x)-u(y)|^p}{|x-y|^{N+sp}}dxdy\right)^{1/p}.$$	

Since $u=0$ in $\Rn\setminus\Om,$ the above integral can be extended to all of $\mathbb{R}^N.$ The embedding
$X_{0,s,p}(\Om)\hookrightarrow L^r(\Om)$ is continuous for any $r\in[1,p^*_s]$ and compact for $r\in[1,p^*_s).$
 Moreover, for $1<p_2 \leq p_1$, $X_{0,s_1,p_1}(\Om)\subset X_{0,s_2,p_2}(\Om)$ (see Lemma 2.2 in Section 2 of \cite{Bhakta-Mukh}).

In the present article, we work with the product space
$$ X=X_{0,s_1,p_1}(\Om) \times X_{0,s_2,p_2}(\Om)
$$
equipped with the norm
$$\|(u,v)\|=\max\{[u]_{s_1,p_1}, [v]_{s_2,p_2}\}$$
such that 
$(X,\|\cdot\|)$ forms a reflexive, separable Banach space. If we consider the norm 
$$\|(u,v)\|_0 = [u]_{s_1,p_1}+ [v]_{s_2,p_2}$$
over the space $X$ then it is easy to verify that $\|\cdot\|_0$ is equivalent to the norm $\|\cdot\|$ over $X$.
We represent the dual space of $X$ as $X^*$ and $\<\cdot,\cdot\>$ denotes the duality pairing of $X$ and $X^*$.

\subsection{Notations and symbols}
In this section, we announce some notations and symbols which will be used all through the paper.
\begin{itemize}
\item
For $i=1,2$, $s_i\in(0,1)$, we define
 $$\dot{W}^{s_i,p_i}(\Rn):=\bigg\{u\in L^{p_i^*}(\Rn) : \displaystyle\int_{\R^{2N}} \frac{|u(x)-u(y)|^p}{|x-y|^{N+s_ip_i}}\,dxdy<\infty\bigg\}$$ and 
\be\lab{S}S_{p_i}=\inf_{u\in\dot{W}^{s_i,p_i}(\Rn)\setminus\{0\}} \frac{\displaystyle\int_{\R^{2N}} \frac{|u(x)-u(y)|^p}{|x-y|^{N+s_ip_i}}dxdy}{\bigg(\displaystyle\int_{\Rn}|u|^{p^*_i}\bigg)^\frac{p}{p^*_i}},\ee
where $p_i^* = \frac{Np_i}{N-s_ip_i}$.
    \item 
 For $r>0, t>0$, we set
\begin{equation*}
\begin{aligned}
&a(t)=\frac{1}{t}-\frac{1}{\al+\ba+2},\;\; b(r,t)=\frac{(r+1)(t-1)}{(\al+\ba+2)(\al+\ba+1)},\\
&c(t)=\frac{(\al+\ba+2-t)}{\al+\ba+1},\;\; d(r,t)=\frac{1}{\f{(\al+1)}{p'_1r^{p'_1}}+\f{(\ba+1)}{p'_2t^{p'_2}}},
\end{aligned}
\end{equation*}
where $p_i^\prime=\frac{p_i}{p_i-1}$ for $i=1,2$. 
\item
We signify
\begin{equation*}
\begin{aligned}
&\eps_1=(\al+1)d(\nu,\mu)\bigg( c(p_1)-\frac{\nu^{p_1}}{p_1}\bigg)\bigg(b(\al,p_1)\min (S_{p_1}^{p^*_1/p_1},  S_{p_2}^{p^*_2/p_2} ) \bigg)^{\frac{p_1}{p_1^*-p_1}}\\
&\eps_2=(\ba+1)d(\nu,\mu)\bigg( c(p_2)-\frac{\mu^{p_2}}{p_2}\bigg)\bigg(b(\ba,p_2)\min (S_{p_1}^{p^*_1/p_1},  S_{p_2}^{p^*_2/p_2} ) \bigg)^{\frac{p_2}{p_2^*-p_2}},
\end{aligned}
\end{equation*}
where $\nu$ and $\mu$ are fixed numbers such that 
$$ 0<\nu< [p_1\,c(p_1)]^{\frac{1}{p_1}},\,\, 0<\mu< [p_2   \, c(p_2)]^{\frac{1}{p_2}}.
$$
\item
 For $i=1,2$, we set
$$\<f_i,u \>_i= \int_\Om f_iu_i~dx.$$
\end{itemize}

\subsection{Hypothesis}
We propose the following hypotheses which will be employed in the entire article.
	\begin{itemize}
		\item [$\textit(H_1)$]
		$(\al+1) \frac{(N-s_1 p_1)}{N p_1} + (\ba+1) \frac{(N-s_2 p_2)}{N p_2}=1$.
		\item[$\textit(H_2)$]
		$\max\{p_1, p_2\}< \al+\ba+2$
		\item[$\textit(H_3)$]
{$0< \|f_1\|_{X^*_{0,s_1,p_1}(\Om)}+\|f_2\|_{X^*_{0,s_2,p_2}(\Om)}<\min \{ \eps_1,\eps_2,1  \}$.}
		\end{itemize}

\subsection{Main result}
Towards the primary objective, we are involved in analyzing 
the following problem
\begin{equation*}
(\mathcal{P}_{f_1,f_2})
\left\{\begin{aligned}
(-\Delta)^{s_1}_{p_1} u  &=u|u|^{\al-1}|v|^{\ba+1}+f_1(x)  \,\mbox{ in }\, \Om, \\
(-\Delta)^{s_2}_{p_2} v  &=|u|^{\al+1}v|v|^{\ba-1}+f_2(x) \,\mbox{ in }\, \Om, \\
u=v &= 0  \,\mbox{ in }\, \Rn \setminus \Om, \\
\end{aligned}
\right.
\end{equation*}
where $\Om \subset \Rn$ is a smooth bounded domain, $s_1,s_2 \in (0,1)$; $f_1$ and $f_2$ are not identically zero on $\Om$. Furthermore, $\al>-1$, $\ba>-1$.

Our goal is to exercise on the existence of solutions for the above model of elliptic system associated with the critical Sobolev exponents in the nonlocal framework. We fetch this result
when $f_1$ and $f_2$ are chosen small in the sense of dual norm. The outcome of our research is planted in the sequential steps:
\begin{itemize}
    \item 
    some fundamental results are obtained;
    \item
    the associated energy functional has a negative infimum on a suitable manifold;
    \item
    a minimization solution to the problem $(\mathcal{P}_{f_1, f_2})$ is achieved;
    \item
    some property of this minimizing solution is also retrieved;
    \item
    the existence of a weak solution for 
    the problem $(P_{f_1, f_2})$ is accomplished.
\end{itemize}

Our main result in this manuscript concerning $(\mathcal P_{f_1,f_2})$ is demonstrated below:
%and $(\mathcal Q_{f_1,f_2})$(refer Section-3) are as follows-
\begin{theorem}\label{thm-main}
	Let $f_i  \in X_{0,s_i,p_i}(\Om)^*$ for $i=1,2$ and $\Om$ be sufficiently smooth bounded domain in $\mathbb R^N$. We furthermore assume that $\textit(H_1), \textit(H_2), \textit(H_3) $ holds true.
	Then there exists $(u^*,v^*) \in \Theta$ which solves  $(\mathcal{P}_{f_1, f_2})$ and satisfies $\mathcal{I}(u^*,v^*)<0$.
\end{theorem}

In the upcoming section, we meticulously take care of the some elementary lemmas, followed by some pivotal results to study our problem $(\mathcal{P}_{f_1, f_2})$ in bounded domain designed in the nonlocal pattern.

\section{ System over bounded domain}
In this section, we move our attention in proving the existence of solutions for the elliptic system $(\mathcal{P}_{f_1, f_2})$ involving critical Sobolev exponents.
Owing to the nonlocal picture, we design the methodology in a very delicate manner using suitable ingredients by a local minimization for an adapted variational problem in the $(s_i,p_i)$-fractional scheme for $i=1,2$.
%\subsection{Functional setting}
 %We first set
%$$\<f_i,u \>_i= \int_\Om f_iu_i~dx$$ for $i=1,2$.
We consider the corresponding energy functional $\mathcal{I}:X \to \mathbb R$ as
\begin{align*}
\mathcal{I}(u_1,u_2)&=\frac{\al+1}{p_1}[u_1]^{p_1}_{s_1,p_1}+\frac{\ba+1}{p_2}[u_2]^{p_2}_{s_2,p_2}-\Iom |u_1|^{\al+1}|u_2|^{\ba+1}\,dx\\
&\quad -(\al+1)\<f_1,u_1\>_1-(\ba+1)\<f_2,u_2\>_2
\end{align*}
for $(u_1,u_2)\in X$.
Let $\Theta \subset X$ be defined as
\begin{gather}\label{Theta}
\Theta:=\big\{(u,v) \in X \setminus \{0\}: \<\mathcal{I}'(u,v), (u,v)\>=0   \big\}
\end{gather}
which gives us that
\begin{equation*}
\begin{aligned}
\mathcal{I}_{|\Theta}(u,v)&=(\al+1)a(p_1)[u]_{s_1,p_1}^{p_1}+(\ba+1)a(p_2)[v]_{s_2,p_2}^{p_2}\\
&\qquad-(\al+1)a(1)\<f_1,u\>_1-(\ba+1) a(1)\<f_2,v\>_2.
\end{aligned}
\end{equation*}
%For $i=1,2$, we denote 
%$$\<f_i,u \>_i= \int_\Om f_iu_i~dx.$$
\subsection{Some Preliminary Lemmas}
To obtain our main result, we need the following preliminary results by a local minimization. 
\begin{lemma}\label{lem-1}
	Let $\Theta$ be defined as in (\ref{Theta}) then $\Theta \neq \emptyset.$
\end{lemma}	
	\begin{proof}
		Let $u_{f_1} \in X_{0,s_1,p_1}(\Om)$ be the unique non-trivial solution of the problem
		\begin{equation*}
		\begin{cases}
		\begin{aligned}
		(-\De)^{s_1}_{p_1}u&=f_1\,\text{ in }\, \Om,\\
		u&=0\,\text{ in }\, \Rn \setminus \Om.
		\end{aligned}
		\end{cases}
		\end{equation*}
Then, $[u]^{p_1}_{0,s_1,p_1}=\<f_1, u_{f_1}\>_1$. Therefore, by the definition of $\mathcal{I}(\cdot,\cdot)$ we have
\begin{equation*}
\<\mathcal{I}'(u_{f_1},0), (u_{f_1},0)\>=(\al+1)[u_{f_1}]_{s_1,p_1}^{p_1}-(\al+1)\<f_1,u_{f_1}\>_1=0.
\end{equation*}		
Hence, $(u_{f_1},0)$ is non-trivial and $(u_{f_1},0) \in \Theta$.		
	\end{proof}	
	
\begin{lemma}\label{lem-2}
	Let $f_i \in X^*_{0,s_i,p_i}$ and $u_{f_i}$ for $i=1,2$ be the  non solution of
		\begin{equation}\label{eq:pblm}
	\begin{cases}
	\begin{aligned}
	(-\De)^{s_i}_{p_i}u&=f_i\,\text{ in }\, \Om,\\
	u&=0\,\text{ in }\, \Rn \setminus \Om.
	\end{aligned}
	\end{cases}
	\end{equation}
Then, we have
$$[u_{f_i}]_{s_i,p_i}^{p_i}=\|f_i\|_{X^*_{0,s_i,p_i}}^{p'_i}$$
where $p_i'=\frac{p_i}{p_i-1}$ is the conjugate of $p_i$.
\end{lemma}		
\begin{proof}
Recall $u_{f_i}$ be the unique solution of (\ref{eq:pblm}) which gives that  for $\Phi_i \in X_{0,s_i,p_i}(\Om) \setminus \{0\}$,
\begin{equation}\label{eq:*1}
\int_{\R^{2N}}\frac{|u_{f_i}(x)-u_{f_i}(y)|^{p_i-2} (u_{f_i(x)}-u_{f_i(y)})(\Phi_i(x)-\Phi_i(y)) }{|x-y|^{N+s_ip_i}}\,dxdy=\<f_i,\Phi_i\>_i.
\end{equation}
Using H\"older's inequality, we note that
\begin{equation*}
\begin{aligned}
\<f_i,\Phi_i\>_i &\leq \int_{\R^{2N}}\frac{| u_{f_i}(x)-u_{f_i}(y)|^{p_i-1} |\Phi_i(x)-\Phi_i(y)|}{|x-y|^{N+s_i p_i}}\,dxdy\\
&\leq\left( \int_{\R^{2N}}\frac{| u_{f_i}(x)-u_{f_i}(y)|^{p_i}}{|x-y|^{N+s_i p_i}}\,dxdy\right)^{\frac{p_i-1}{p_i}}
\left(\int_{\R^{2N}}\frac{ |\Phi_i(x)-\Phi_i(y)|^{p_i}  }{|x-y|^{N+s_i p_i}}\,dxdy \right)^{\frac{1}{p_i}}\\
&=[u_{f_i}]_{ s_i,p_i}^{p_i-1}[\Phi_i]_{s_i,p_i}.
\end{aligned}
\end{equation*}
This implies,
\begin{gather*}
\Bigg| \frac{\<f_i,\Phi_i\>_i }{[\Phi]_{s_i,p_i}}\Bigg| \leq [u_{f_i}]_{s_i,p_i}^{p_i-1},
\end{gather*}
which immediately yields us
\begin{gather}\label{eq:*2}
\|f_i\|_{X^*_{0,s_i,p_i}}^{p'_i} \leq [u_{f_i}]_{s_i,p_i}^{p_i}.
\end{gather}
 On the other hand, taking $\Phi_i=u_{f_i}$ in (\ref{eq:*1}) we get,
\begin{gather}\label{eq:*3}
[u_{f_i}]_{s_i,p_i}^{p_i}\leq \|f_i\|^{p'_i}_{X^*_{0,s_i,p_i}}.
\end{gather}
Combining (\ref{eq:*2}) and (\ref{eq:*3}) we come by the following conclusion,
\begin{gather*}
[u_{f_i}]_{s_i,p_i}^{p_i}=\|f_i\|_{X^*_{0, s_i, p_i}}^{p'_i}.
\end{gather*}
 This winds-up the proof of the Lemma.
\end{proof}	
The subsequent result is an easy consequence of the Ekeland Variational Principle.
\begin{lemma}\label{lem-3}
	Let $\al,\ba$ satisfy $\al+\ba+2> \max\{p_1, p_2\}$. Then there exists a sequence $\{(u_k,v_k) \}_{k \in \N}$
such that 
\begin{gather}\label{eq:I}
\inf_{(u,v) \in \Theta} \mathcal{I}(u,v) < \mathcal{I}(u_k,v_k) <\inf_{(u,v) \in \Theta} \mathcal{I}(u,v)+\frac{1}{k},
\end{gather}
\begin{gather}\lab{eq:I'}
\|\mathcal{I}'_{|_{\Theta}}(u_k,v_k)\|_{X^*} \leq \frac{1}{k} \,\text{ for all }\, k \in \N.
\end{gather}
\end{lemma}		
\begin{proof}
Using H\"older's inequality and Young's inequality we get,
\begin{equation*}
\begin{aligned}
\|\mathcal{I}'_{|_{\Theta}}(u,v) \| &\geq (\al+1)a(p_1)[u]_{s_1,p_1}^{p_1}+(\ba+1)a(p_2)[v]_{s_2,p_2}^{p_2}\\
&\qquad-(\al+1)\theta_1^{p_1}[u]_{s_1,p_1}^{p_1}-(\al+1)[a(1){\|f_1\|_{X^*_{0,s_1,p_1}}^{p_1}} ]^{p'_1}\theta_1^{-p'_1}\\
&\qquad -(\ba+1)\theta_2^{p_2}[v]_{s_2,p_2}^{p_2}-(\ba+1)[a(1){\|f_2\|_{X^*_{0,s_2,p_2}}}]^{p'_2}\\
&\geq (\al+1) [u]_{s_1,p_1}^{p_1}\left(a(p_1) -\theta_1^{p_1}\right) +(\ba+1) [v]_{s_2,p_2}^{p_2}\left(a(p_2) -\theta_2^{p_2}\right)\\
&\qquad-\theta_1^{-p'_1}(\al+1) \left(a(1) \|f_1\|_{X^*_{0,s_1,p_1}} \right)^{p'_1}-\theta_2^{-p'_2}(\ba+1) \left(a(1) \|f_2\|_{X^*_{0,s_2,p_2}} \right)^{p'_2}\\
& \geq -\theta_1^{-p'_1}(\al+1) \left(a(1) \|f_1\|_{X^*_{0,s_1,p_1}} \right)^{p'_1}-\theta_2^{-p'_2}(\ba+1) \left(a(1) \|f_2\|_{X^*_{0,s_2,p_2}} \right)^{p'_2}
\end{aligned}
\end{equation*}	
by choosing $\theta_1, \theta_2>0$ small enough. Thus $\mathcal{I}(\cdot,\cdot)$ is bounded below on $\Theta$. Now applying Ekeland Variational principle, we assert that there exists a sequence $\{(u_k,v_k) \}_{k \in \N} \subset \Theta$ such that (\ref{eq:I}) and (\ref{eq:I'}) hold. This concludes the proof of the Lemma.
\end{proof}		
\subsection{Few technical lemmas}
In this section, we are dealing with some technical lemmas in order to obtain our main result.
Let us denote 
 $$ m:=\inf_{(u,v) \in \Theta} \mathcal{I}(u,v).$$
\begin{lemma}\label{lem-4}
It holds that
\begin{gather*}
m<\min \left\{ -\frac{\al+1}{p'_1} \|f_1\|^{p'_1}_{X^*_{0,s_1,p_1}},\, -\frac{\ba+1}{p'_2} \|f_2\|^{p'_2}_{X^*_{0,s_2,p_2}} \right\}.
\end{gather*}
\end{lemma}		
	\begin{proof}
	Let $u_{f_i}$ be unique solution of (\ref{eq:pblm}) given in Lemma  \ref{lem-2}. Then, by Lemma \ref{lem-2} we have
	\begin{equation*}
	\begin{aligned}
	\mathcal{I}(u_{f_1},0)&=(\al+1) \left( \frac{1}{p_1}[u_{f_1}]_{s_1,p_1}^{p_1}-\<f_1, u_{f_1}\>_1\right)=-\frac{(\al+1)}{p'_1}\|f_1\|_{X^*_{0,s_1,p_1}}^{p'_1}.
	\end{aligned}
	\end{equation*}
	In a similar fashion, we have
			$$\mathcal{I}(0,u_{f_2})=-\frac{(\ba+1)}{p'_2}\|f_2\|_{X^*_{0,s_2,p_2}}^{p'_2}.$$
		Since $m< \mathcal{I}(u_{f_1},0)$ and $m<\mathcal{I}(0, u_{f_2})$, so we accomplish the result.
	\end{proof}	
	
Let us define $\mathcal{K}: X \to \R$ defined by
\begin{equation}\label{eq:K}
\begin{aligned}
\mathcal{K}(u,v)&=\< \mathcal{I}'(u,v), (u,v)\>\\
&=(\al+1) [u]_{s_1,p_1}^{p_1}+(\ba+1)[v]_{s_2,p_2}-(\al+\ba+2)\Iom |u|^{\al+1}|v|^{\ba+1}\,dx\\
&\qquad-(\al+1)\<f_1,u\>_1-(\ba+1)\<f_2,v\>_2,
\end{aligned}
\end{equation}
for $(u,v) \in X$. Following result ensures that every critical point of $\mathcal K$ is non trivial.

\begin{lemma}\label{lem-5}
	Let $\mathcal{K}$ be defined in (\ref{eq:K}) above. Then, $\mathcal{K}'(u,v)=0$ for $(u,v) \in \Theta$ implies that $u \not\equiv0 \not\equiv v$.
\end{lemma}	

\begin{proof}
We will use the method by contradiction for proof. So we assume $(u_1,v_1) \in \Theta$ be such that $\mathcal{K'}(u_1,v_1)=0$ and either $u_1\equiv 0$ or $v_1\equiv 0$. Then $(u_1, v_1) \in \Theta$ implies that $(u_1,v_1) \neq (0,0)$, by definition of $\Theta$. Without loss of generality, let $v_1\equiv 0$. As $(u_1, v_1) \in \Theta$, so
$\< \mathcal{I}'(u_1,v_1), (u_1,v_1)\>=0$. This implies $\mathcal{K}(u_1, v_1)=0$ which asserts that 
\begin{gather}\label{eq:a}
(\al+1)[u_1]_{s_1,p_1}^{p_1}-(\al+1)\<f_1,u_1\>_1=0
\end{gather}
since $v_1\equiv 0$.
Again as $\mathcal{K'}(u_1,v_1)=0$, so we get $\< \mathcal{K'}(u_1,v_1), (u_1,v_1)\>=0$ which yields us
\begin{gather}\label{eq:b}
p_1(\al+1)[u_1]_{s_1,p_1}^{p_1}-(\al+1)\<f_1,u_1\>_1=0.
\end{gather}
Equations (\ref{eq:a}) and (\ref{eq:b}) together imply
\begin{gather*}
(p_1-1)(\al+1)[u_1]_{s_1,p_1}^{p_1}=0,
\end{gather*}
that is $u_1\equiv 0$ a.e. in $\Om$. Hence, $(u_1,v_1)=(0,0) \in \Theta$ which is a contradiction to the definition of $\Theta$. This completes the proof.
\end{proof}	
The next outcome is an important consequence of Lemma \ref{lem-5}.
\begin{lemma}\label{lem-6}
	Let $(f_1, f_2) \in X^* \setminus \{ (0,0)\}$ be such that
	\begin{gather}\label{eq:E}
	 0< \|f_1\|_{X^*_{0,s_1,p_1}}+\|f_2\|_{X^*_{0,s_2,p_2}}<\min\{\eps_1, \eps_2,1 \}.
	\end{gather}
	Then, $\mathcal{K'}(u,v)\neq 0$ for every $(u,v) \in \Theta$.
\end{lemma}
\begin{proof}
	We will prove this result by the method of contradiction. Let us suppose that
	\begin{gather*}
	\mathcal{K}'(\tilde{u}, \tilde{v}) =0 \, \text{ for some }\, (\tilde{u}, \tilde{v}) \in \Theta.
	\end{gather*}
	By Lemma \ref{lem-5}, we know that  $\tilde{u} \not \equiv 0\not \equiv \tilde{v}$. 
Whereas 
$\< 	\mathcal{K}'(\tilde{u}, \tilde{v}), (\tilde{u},\tilde{v}) \> =0
$	implies
\begin{equation}\label{eq:c}
\begin{aligned}
&(\al+1) p_1[\tilde{u}]_{s_1,p_1}^{p_1}+(\ba+1)p_2[\tilde{v}]_{s_2, p_2}^{p_2} \\
&-(\al +\ba +2)^2 \Iom |\tilde{u}|^{\al+1}|\tilde{v}|^{\ba+1}\,dx-(\al+1)\<f_1,\tilde{u}\>_1-(\ba+1)\<f_2, \tilde{v}\>_2=0.
\end{aligned}
\end{equation}
Again as $(\tilde{u}, \tilde{v}) \in \Theta$, so we have $\< \mathcal{I}'(\tilde{u}, \tilde{v}), (\tilde{u}, \tilde{v})\>=0$, that is,
\begin{equation}\label{eq:d}
\begin{aligned}
&(\al+1) [\tilde{u}]_{s_1,p_1}^{p_1}+(\ba+1)[\tilde{v}]_{s_2, p_2}^{p_2} \\
&-(\al +\ba +2) \Iom |\tilde{u}|^{\al+1}|\tilde{v}|^{\ba+1}\,dx-(\al+1)\<f_1, \tilde{u}\>_1-(\ba+1)\<f_2, \tilde{v}\>_2 =0.
\end{aligned}
\end{equation}
Equations (\ref{eq:c}) and (\ref{eq:d}) together yields,
\begin{equation*}
\begin{aligned}
(\al+1) (p_1-1 )[\tilde{u}]_{s_1,p_1}^{p_1}+(\ba+1)(p_2-1)[\tilde{v}]_{s_2, p_2}^{p_2} 
=(\al +\ba +2)(\al+\ba+1) \Iom |\tilde{u}|^{\al+1}|\tilde{v}|^{\ba+1}\,dx.
\end{aligned}
\end{equation*}
Using Young's and Sobolev inequality we have,
\begin{equation*}
\begin{aligned}
&\frac{(\al+1)(p_1-1)}{(\al+\ba+2)(\al+\ba+1)}[ \tilde{u} ]_{s_1, p_1}^{p_1}+\frac{(\ba+1)(p_2-1)}{(\al+\ba+2)(\al+\ba+1)}[ \tilde{v} ]_{s_2, p_2}^{p_2}\\
&=\Iom | \tilde{u}|^{\al+1}|\tilde{v}|^{\ba+1}\,dx \leq \frac{(\al+1)}{p_1^*}\big(\Iom |\tilde{u}|^{p_1^*} \,dx \big)^{\frac{1}{p_1^*}}+\frac{(\ba+1)}{p_2^*}\big(\Iom |\tilde{v}|^{p_2^*} \,dx \big)^{\frac{1}{p_2^*}}\\
&=\frac{(\al+1)}{p_1^*} |\tilde{u}|_{L^{p_1^*}(\Om)}+\frac{(\ba+1)}{p_2^*} |\tilde{v}|_{L^{p_2^*}(\Om)}\\
&\leq \frac{(\al+1)}{p_1^*}\frac{1}{S_{p_1}^{\frac{1}{p_1}}}
 [\tilde{u}]_{s_1,p_1}+\frac{(\ba+1)}{p_2^*}\frac{1}{S_{p_2}^{\frac{1}{p_2}}}
 [\tilde{v}]_{s_2,p_2},
\end{aligned}
\end{equation*}
where we have used 
$$\frac{\al+1}{p_1^*}+\frac{\ba+1}{p_2^*}=1
\,\text{ and }\,
S_{p_i}^{\frac{1}{p_i}}|u|_{L^{p_i}} \leq [u]_{s_i,p_i}.$$ 
Now without loss of generality, we may assume that $$ [\tilde{u}]_{s_1, p_1}^{p_1^*} \leq [ \tilde{v}]_{s_2, p_2}^{p_2^*}.$$
	Hence we get,
	\begin{equation*}
	\begin{aligned}
	b(\al, p_1) [ \tilde{u}]_{s_1,p_1}^{p_1}+b(\ba,p_2) [ \tilde{v}]_{s_2,p_2}^{p_2}
	\leq \frac{(\al+1)}{p_1^*}\frac{1}{S_{p_1}^{p_1^*/p_1} }[ \tilde{u}]_{s_1, p_1}^{p_1^*}+\frac{(\ba+1)}{p_2^*}\frac{1}{S_{p_2}^{p_2^*/p_2} }[ \tilde{v}]_{s_2, p_2}^{p_2^*},
	\end{aligned}
	\end{equation*}
	which implies that 
	\begin{equation*}
	\begin{aligned}
	b(\ba,p_2) [ \tilde{v}]_{s_2,p_2}^{p_2}
	\leq \left(\frac{\al+1}{p_1^*} + \frac{\ba+1}{p_2^*}\right)\frac{1}{\min\{S_{p_1}^{p_1^*/p_1}, S_{p_2}^{p_2^*/p_2}\}}[ \tilde{v}]_{s_2, p_2}^{p_2^*}
		\end{aligned}
	\end{equation*}
	This yields us,
	\begin{equation}\label{eq:M}
	M :=b(\ba, p_2) \big( \min\{S_{p_1}^{p_1^*/p_1} S_{p_2}^{p_2^*/p_2}\}\big)^{\frac{1}{p_2^*-p_2}}\leq [ \tilde{v}]_{s_2,p_2}.
	\end{equation}
 On the other hand, if we have assumed 
	$$ [\tilde{u}]_{s_1, p_1}^{p_1^*} \geq [ \tilde{v}]_{s_2, p_2}^{p_2^*},$$ then we would get
	\begin{equation}\label{eq:M-}
	M \leq [\tilde{u}]_{s_1,p_1}.
	\end{equation}
	Multiplying (\ref{eq:d}) by $(\al+\ba+2)$ and then subtracting from (\ref{eq:c}), we get using H\"older and Young's inequality,
	\begin{equation*}
	\begin{aligned}
	&(\al+1)c(p_1)[\tilde{u}]_{s_1, p_1}^{p_1}+(\ba+1)c(p_2)[\tilde{v}]_{s_2, p_2}^{p_2}\\
	&= (\al+1) \<f_1, \tilde{u}\>+(\ba+1)\<f_2, \tilde{v}\>\\
	&\leq (\al+1)\frac{\de_1^{p_1}}{p_1}[\tilde{u}]_{s_1, p_1}^{p_1}+
	\frac{(\al+1)}{\de_1^{p'_1}p'_1}\|f_1\|_{X^*_{0,s_1,p_1}}^{p'_1}+
	(\ba+1)\frac{\de_2^{p_2}}{p_2}[ \tilde{v}]_{s_2, p_2}^{p_2}+\frac{(\ba+1)}{\de_2^{p'_2}p'_2}\|f_2\|_{X^*_{0,s_2,p_2}}^{p'_2},
	\end{aligned}
	\end{equation*}
	for some $\de_1, \de_2>0$. Choosing $0<\de_1< \big(p_1 \, c(p_1)\big)^{\frac{1}{p_1}}$ and $0<\de_2< \big(p_2 \, c(p_2)\big)^{\frac{1}{p_2}}$, we get
	\begin{equation}\label{eq:est}
	\begin{aligned}
	&(\al+1) \left( c(p_1)-\frac{\de_1^{p_1}}{p_1} \right) [\tilde{u}]_{s_1, p_1}^{p_1}+
	(\ba+1) \left( c(p_2)-\frac{\de_2^{p_2}}{p_2} \right) [\tilde{v}]_{s_2, p_2}^{p_2}\\
	&\qquad \leq \frac{(\al+1)}{\de_1^{p'_1}p'_1}\|f_1\|_{X^*_{0,s_1,p_1}}^{p'_1}+\frac{(\ba+1)}{\de_2^{p'_2}p'_2}\|f_2\|_{X^*_{0,s_2,p_2}}^{p'_2}.
	\end{aligned}
	\end{equation}
	Hence, (\ref{eq:est}) yields us
		\begin{equation}\label{eq:est-i}
	(\al+1) \left( c(p_1)-\frac{\de_1^{p_1}}{p_1} \right)[ \tilde{u}]_{s_1, p_1}^{p_1}
	 \leq \frac{(\al+1)}{\de_1^{p'_1}p'_1}\|f_1\|_{X^*_{0,s_1,p_1}}^{p'_1}+\frac{(\ba+1)}{\de_2^{p'_2}p'_2}\|f_2\|_{X^*_{0,s_2,p_2}}^{p'_2},
	\end{equation}
	and
		\begin{equation}\label{eq:est-ii}
	(\ba+1) \left( c(p_2)-\frac{\de_2^{p_2}}{p_2} \right)[\tilde{v}]_{s_2, p_2}^{p_2} \leq \frac{(\al+1)}{\de_1^{p'_1}p'_1}\|f_1\|_{X^*_{0,s_1,p_1}}^{p'_1}+\frac{(\ba+1)}{\de_2^{p'_2}p'_2}\|f_2\|_{X^*_{0,s_2,p_2}}^{p'_2}.
	\end{equation}
	Using (\ref{eq:M}), we get from (\ref{eq:est-ii}) that,
	\begin{equation*}
	(\ba+1) d(\de_1, \de_2) \left(c(p_2)-\frac{\de_2^{p_2}}{p_2} \right)
	b(\ba, p_2) \Big(\min\{ S_{p_1}^{p_1^*/p_1}, S_{p_2}^{p_2^*/p_2}\}\Big)^{\frac{p_2}{p_2^*-p_2}}
	<\|f_1\|_{X^*_{0,s_1,p_1}}^{p'_1}+\|f_2\|_{X^*_{0,s_2,p_2}}^{p'_2}.
	\end{equation*}
	By similar analysis, using (\ref{eq:M-}), we would obtain,
		\begin{equation*}
	(\al+1) d(\de_1, \de_2) \left( c(p_1)-\frac{\de_1^{p_1}}{p_1} \right)
	b(\al, p_1) \Big( \min\{S_{p_1}^{p_1^*/p_1}, S_{p_2}^{p_2^*/p_2}\}\Big)^{\frac{p_1}{p_1^*-p_1}}
	<\|f_1\|_{X^*_{0,s_1,p_1}}^{p'_1}+\|f_2\|_{X^*_{0,s_2,p_2}}^{p'_2}.
	\end{equation*}
	Consequently, we get
	\begin{gather*}
	\min\{\eps_1, \eps_2 \}< \|f_1\|_{X^*_{0,s_1,p_1}}^{p'_1}+\|f_2\|_{X^*_{0,s_2,p_2}}^{p'_2}
	\end{gather*}
	which is a contradiction to (\ref{eq:E}), bearing out the proof of the lemma.
\end{proof}	
The ensuing results below treat some property of the minimizing sequence.
\begin{lemma}\label{lem-7}
	Let $\frac{\al+1}{p_1^*}+\frac{\ba+1}{p_2^*}=1$ and $m$ be defined as in Lemma \ref{lem-4}. Let $\{(u_k,v_k) \}_{k \in \N} \subset \Theta$ be such that $\lim_{k \to \infty} \mathcal{I}(u_k,v_k)=m$. Then, there exists $(u^*,v^*) \in X$ such that $(u_k,v_k) \deb (u^*, v^*)$ weakly in $X$.
\end{lemma}	
\begin{proof}
Using Lemma \ref{lem-3}, we have,
\begin{equation*}
\begin{aligned}
-m-\frac{1}{k}\| (u_k, v_k) \|_X& \leq \< \mathcal{I}'_{\Theta}(u_k, v_k), (u_k, v_k) \>-\mathcal{I}_{\Theta}(u_k, v_k)\\
&\qquad \leq m+\frac{1}{k}\| (u_k, v_k)\|_X.
\end{aligned}
\end{equation*}	
This gives 
\begin{equation*}
\begin{aligned}
-m-\frac{1}{k}\| (u_k, v_k) \|_X& \leq \frac{\al+1}{p_1^*}\left(1-\frac{p_1}{\al+\beta+1}\right)[u_k]_{s_1, p_1}^{p_1} + \frac{\beta+1}{p_2^*}\left(1-\frac{p_2}{\al+\beta+1}\right)[v_k]_{s_2, p_2}^{p_2}\\
& \leq m+\frac{1}{k}\| (u_k, v_k)\|_X
\end{aligned}
\end{equation*}
which gives a contradiction if $\| (u_k, v_k)\|_X\to \infty$ as $k\to \infty$.
{Therefore $\{(u_k, v_k) \}_{k \in \N}$ is bounded in $X$.}
	Since $X$ is a reflexive Banach space, there exists $(u^*, v^*) \in X$ such that up to a subsequence, $(u_k, v_k) \deb (u^*, v^*)$ weakly in $X$ as $k \to \infty$, thereby completing the proof.
\end{proof}

\begin{proposition}\label{prop}
	Let $\de_1, \de_2 \in \R$ be such that $0<\de_i< \big( p_i\, c(p_i)\big)^{\frac{1}{p_i}}$ and $f_i \in X^*_{0, s_i, p_i}$ for $i=1,2$ satisfying
	$$ 0< \|f_1\|_{X^*_{0, s_1, p_1}}+\|f_2\|_{X^*_{0, s_2, p_2}}< \min\{\eps_1, \eps_2, 1\}. 
	$$
	Then there exists $\de_0>0$ such that 
	$$ \Big| \<\mathcal{K}'(u_k, v_k), (u_k, v_k)\>\Big| \geq \de_0>0,
	$$
	for any minimizing sequence $\{(u_k, v_k)\}_{k\in \N}$ of $\mathcal{I}$ in $\Theta$.
\end{proposition}		
\begin{proof}
 By contradiction, suppose the assertion is not true. Then, there exists a minimizing sequence $\{(u_k, v_k)\}_{k \in \N} \subset \Theta$ of $\mathcal{I}$ such that
	\begin{equation}\label{new-1}
	    \Big| \< \mathcal{K'}(u_k, v_k), (u_k, v_k) \>\Big| \to 0 \text{ as }\, k \to \infty.
	\end{equation}
	Let $t_k:=\< \mathcal{K'}(u_k, v_k), (u_k, v_k) \> \, \text{for all}\, k \in \N$. Then by \eqref{new-1}, we have
	\begin{equation}\label{eq:t_k}
	\lim_{k \to \infty} |t_k|=0,\,\text{ that is }\, \lim_{k \to \infty}t_k=0.
	\end{equation}	
	Since $(u_k, v_k) \subset \Theta$, we have
 \begin{equation}\label{eq:^}
	\begin{aligned}
	&(\al+1) [u_k]_{s_1, p_1}^{p_1}+(\ba+1) [ v_k]_{s_2, p_2}^{p_2}-(\al+\ba+2) \Iom |u_k|^{\al+1} |v_k|^{\ba+1}\,dx\\
	&\qquad -(\al+1) \<f_1, u_k\>_1-(\ba+1)\<f_2, v_k\>_2=0.
	\end{aligned}
	\end{equation}
	Combining definition of $t_k$ with (\ref{eq:^}) we have,
	\begin{equation}\label{eq:+1}
	\begin{aligned}
	&	(\al+1)(p_1-1)[ u_k]_{s_1, p_1}^{p_1}+(\ba+1)(p_2-1) [v_k]_{s_2, p_2}^{p_2}\\
		&\qquad=(\al+\ba+2)(\al+\ba+1) \Iom |u_k|^{\al+1} |v_k|^{	\ba+1}\,dx+t_k.
	\end{aligned}
	\end{equation}	
	By similar argument as in the proof of Lemma \ref{lem-6}, we obtain assuming $ [u_k]_{s_1, p_1}^{p_1^*} \leq [v_k]_{s_2, p_2}^{p_2^*}$,	\begin{equation*}
\min\{S_{p_1}^{p_1^*/p_1} , S_{p_2}^{p_2^*/p_2}\}\left( b(\ba, p_2)-\frac{t_k}{[v_k]_{s_2, p_2}^{p_2}}\right)
	\leq [v_k]_{s_2, p_2}^{p_2^*-p_2}.
	\end{equation*}	
Therefore, we conclude using Lemma \ref{lem-4} that
$$\frac{1}{[v_k]_{s_2, p_2}} \leq M_0 \,\text{ for some }\, M_0>0.$$
Thus, we have
$$\frac{t_k}{[v_k]_{s_2, p_2}^{p_2}}< b(\ba, p_2) \,\text{ for some }\, k \,\text{ large enough. }\, 
$$
This yields
\begin{gather*}
\Big(  b(\ba, p_2) \min\{S_{p_1}^{p_1^*/p_1},S_{p_2}^{p_2^*/p_2} \} \Big)^{\frac{1}{p_2^*-p_2}}-At_k \leq [v_k]_{s_2, p_2},
\end{gather*}
where $A$ is a constant depending on $\al, \ba, p_i, p_i^*, S_{p_i}$, $i=1,2$. As a result, we have for 
$0<\de_i<\big(p_i\, c(p_i) \big)^{\frac{1}{p_i}}$
\begin{equation*}
\begin{aligned}
&(\al+1) d(\de_1, \de_2) \left(c(p_1)-\frac{\de_1^{p_1}}{p_1} \right)	\Big(  b(\al, p_1) \min\{S_{p_1}^{p_1^*/p_1},S_{p_2}^{p_2^*/p_2} \} \Big)^{\frac{p_1}{p_1^*-p_1}}-A t_k\\
& \leq \|f_1\|_{X^*_{0,s_1, p_1}}+\|f_2\|_{X^*_{0,s_2, p_2}}+t_k,
\end{aligned}
\end{equation*}		
and
\begin{equation*}
	\begin{aligned}
		&(\ba+1) d(\de_1, \de_2) \left( c(p_2)-\frac{\de_2^{p_2}}{p_1} \right)	\Big(  b(\beta, p_2) \min\{S_{p_1}^{p_1^*/p_1}, S_{p_2}^{p_2^*/p_2}\}  \Big)^{\frac{p_2}{p_2^*-p_2}}-B t_k\\
		& \leq \|f_1\|_{X^*_{0,s_1, p_1}}+\|f_2\|_{X^*_{0,s_2, p_2}}+t_k.
	\end{aligned}
\end{equation*}	
Letting $k \to \infty$ we get, 
\begin{equation*}
( \al +1) d(\de_1, \de_2) \left( c(p_1)-\frac{\de_1^{p_1}}{p_1} \right) \Big(b(\al, p_1)\, \min\{S_{p_1}^{p_1^*/p_1}, S_{p_2}^{p_2^*/p_2}\} \Big)^{\frac{p_1}{p_1^*-p_1}} \leq \|f_1\|_{X^*_{0,s_1,p_1}}+\|f_2\|_{X^*_{0,s_2,p_2}}
\end{equation*}	
and
\begin{equation*}
	( \ba +1) d(\de_1, \de_2) \left( c(p_2)-\frac{\de_2^{p_2}}{p_2} \right) \Big(b(\ba, p_2)\, \min\{S_{p_1}^{p_1^*/p_1},S_{p_2}^{p_2^*/p_2}\} \Big)^{\frac{p_2}{p_2^*-p_2}} \leq \|f_1\|_{X^*_{0,s_1,p_1}}+\|f_2\|_{X^*_{0,s_2,p_2}}.
\end{equation*}	
This contradicts the hypothesis of the proposition, thereby winding-up the proof.
\end{proof}	

\subsection{Proof of Theorem \ref{thm-main}}

Taking into account the minimizing sequence $\{ (u_k, v_k)\}_{k \in \N} \subset \Theta$, we write,
$$\mathcal{I}'(u_k, v_k)=\mathcal{I}'|_{\Theta}(u_k, v_k)-\la_k \mathcal{K}'(u_k, v_k),
$$ that is,
\begin{gather*}
\<	\mathcal{I}'(u_k, v_k), (u_k, v_k) \>=\< \mathcal{I}'|_{\Theta}(u_k, v_k), (u_k, v_k) \>-\la_k\< \mathcal{K}'(u_k, v_k),(u_k, v_k)\>.
\end{gather*}	
We notice that $\< \mathcal{I}'(u_k, v_k), (u_k, v_k)\>=0$ on L.H.S., as $\{ (u_k, v_k)\}_{k \in \N} \subset \Theta$.
On R.H.S., the sequence $\{(u_k, v_k) \}_{k \in \N}$ be such that,
$$ \mathcal{I'}|_{\Theta}(u_k, v_k) \to 0\,\text{ as }\, k \to \infty.
$$
This implies,
$$ \<\mathcal{I'}|_{\Theta}(u_k, v_k), (u_k, v_k)\> \to 0\,\text{ as }\, k \to \infty.
$$
Using Proposition \ref{prop}, we conclude that $\la_k \to 0$. As a consequence, we notice that,
\begin{gather*}
\lim_{k \to \infty} \< \mathcal{I}'(u_k, v_k), (\phi_1, \phi_2)\>=0 \;\;\text{for all}\, (\phi_1, \phi_2)\in X.
\end{gather*}	
Hence, we obtain
\begin{equation*}
	\begin{aligned}
\begin{cases}
	(-\De)^{s_1}_{p_1} u_k&=u_k|u_k|^{\al-1}|v_k|^{\ba+1}+f_1+f_k\,\text{ in }\, \Om,\\
	(-\De)^{s_2}_{p_2} v_k&=|u_k|^{\al+1}v_k |v_k|^{\ba-1}+f_2+g_k\,\text{ in }\, \Om,\\
	&u_k=v_k=0\;\text{in}\; \mathbb R^N \setminus\Om,
\end{cases}	
\end{aligned}
\end{equation*}
where $\{f_k\}$ and $\{g_k\}$ denote two sequences converging strongly to $0$ in $(X_{0,s_i, p_i}(\Om))^*$ respectively for $i=1,2$. Let us define
$\{M_k\}_{k \in \N}$ and $\{N_k\}_{k_n}$ be such that
$M_k=u_k|u_k|^{\al-1}|v_k|^{\ba+1}$ and $N_k=|u_k|^{\al+1} v_k|v_k|^{\ba-1}$. We now approach towards the following proposition regarding $\{M_k\}$ and $\{N_k\}$.

\begin{proposition}\label{prop-bounded}
The sequences $\{M_k\}_{k \in \N}$ and $\{N_k\}_{k \in \N}$ defined above are  bounded in $X^*_{0, s_1, p_1}(\Om)$ and $X^*_{0, s_2, p_2}(\Om)$ respectively.	
\end{proposition}	

\begin{proof}
We know that the embedding $X_{0, s_i, p_i}(\Om) \hookrightarrow L^{p_i^*}(\Om)$ and
 $L^{(p_i^*)'}(\Om) \hookrightarrow X^*_{0, s_i, p_i}(\Om)$ are continuous. 
 Our aim is to deduce that $\{M_k\}_{k \in \N}$ and $\{N_k\}_{k \in \N}$ are bounded in $L^{(p_i^*)'}(\Om)$ respectively for $i=1,2$.
To conclude our result, it is sufficient to consider $\{M_k\}_{k \in \N}$.
Using the fact that $\{u_k\}_{k \in \N}$ and $\{v_k\}_{k \in \N}$ are bounded respectively in $X_{0, s_i, p_i}(\Om)$ for $i=1,2$ (refer Lemma \ref{lem-7}) and embedding results, we have the following estimate:
for every $\psi \in L^{p_1^*}(\Om)$, there exists a constant $\mathcal{F}'>0$ such that,
\begin{equation}\label{eq:b-1}
\begin{aligned}
\Big| \Iom M_k \psi \, dx \Big| &\leq \Iom |M_k||\psi|\,dx \leq \Iom |u_k|^{\al}|v_k|^{\ba+1}|\psi|\,dx\\
&\leq |u_k|_{L^{p_1^*}(\Om)}^{\al} |v_k|_{L^{p_2^*}(\Om)}^{\ba+1}|\psi|_{L^{p_1^*}(\Om)} \leq \mathcal{F}'|\psi|_{L^{p_1^*}(\Om)}.
\end{aligned}
\end{equation}
We note that (\ref{eq:b-1}) is valid due to H\"older's inequality as,
$$ \frac{\al+1}{p_1^*}+\frac{\ba+1}{p_2^*}=1.
$$
Using (\ref{eq:b-1}), for every $\psi \in L^{(p_1^*)'}(\Om) \setminus \{0\}$, we have,
\begin{equation*}
	\frac{\left| \Iom M_k \psi \, dx \right|}{\|\psi\|_{L^{p_1^*}}} \leq \mathcal{F}'.
\end{equation*}	
This implies,
\begin{equation*}
\sup_{\psi \in L^{p_1^*}(\Om) \setminus \{0\}  }\frac{\Big| \Iom M_k \psi \, dx \Big|}{|\psi|_{L^{p_1^*}(\Om)}} =|M_k|_{L^{(p_1^*)'}(\Om)}
\leq \mathcal{F}'.
\end{equation*}	
We conclude that $\{M_k\}_{k \in \N}$ is bounded in $L^{(p_i^*)'}(\Om)$ and thereupon, is bounded in $X^*_{0, s_1, p_1}(\Om)$. This completes the proof.	
	
\end{proof}	

\begin{proposition}\label{prop-bounded1}
	The sequences $\{M_k\}_{k \in \N}$ and $\{N_k\}_{k \in \N}$ are bounded in $L^1(\Om)$. Moreover, we have,
	\begin{equation*}
		\lim_{k \to \infty} \Iom M_k \phi_1\,dx=\Iom M^*\phi_1\,dx,\,\forall\, \phi_1 \in L^{\infty}(\Om),
	\end{equation*}	 
and
\begin{equation*}
	\lim_{k \to \infty} \Iom N_k \phi_2 \,dx=\Iom N^*\phi_2\,dx,\,\forall\, \phi_2 \in L^{\infty}(\Om),
\end{equation*}	 
where $M^*=u^*|u^*|^{\al-1}|v^*|^{\ba+1}$, $N^*=|u^*|^{\al+1}v^*|v^*|^{\ba-1}$ and $(u^*,v^*)$ is as defined in Lemma \ref{lem-7}.
\end{proposition}	
\begin{proof}
	Using H\"older inequality, we can see that $\{M_k\}_{k \in \N}$ and $\{N_k\}_{k \in \N}$ are bounded in $L^1(\Om)$. Using Lemma \ref{lem-7}, we obtain the existence of $u^*$ and $v^*$ which furthermore guarantees that some subsequences denoted again by $\{M_k\}_{k \in \N}$ and $\{N_k\}_{k \in \N}$ converge a.e. in $\Om$ to $M^*$ and $N^*$ respectively. In a similar fashion, we can show that $M^*$ and $N^*$ belong to $L^1(\Om)$. Let 
	$$|N_k-N^*|_{L^1(\Om)}< (\mathcal{F}')^{R_{\alpha\beta}}$$ where $R_{\alpha\beta}>0$ is such that
	$$\frac{1}{R_{\al \ba}}=\frac{\al}{p_1^*}+\frac{\ba+1}{p_2^*}.$$
	We will now apply Egoroff's Theorem \cite{Brezis-83}. For fix $\eps>0$, there exists a measurable subset $\mathcal{A}$ of $\Om$ such that 
	\begin{gather*}
		\big| \Om \setminus \mathcal{A} \big|^{\frac{1}{p_1^*}}< \frac{\eps}{2\mathcal{F}'}.
	\end{gather*}	
and the sequence $\{N_k \}_{k \in \N}$ converges uniformly to $N^*$ in $\mathcal{A}$. As a consequence, we obtain,
\begin{gather*}
	\Iom |N_k-N^*|\,dx=\int_{\Om \setminus \mathcal{A}}|N_k-N^*|\,dx+\int_{\mathcal{A}}|N_k-N^*|\,dx.
\end{gather*}	
Using uniform convergence of $N_k$ to $N^*$ in $\mathcal{A}$, we note that there exists $k_\eps>0$ such that for $k \geq k_\eps$, we have,
\begin{gather*}
	|N_k(x)-N^*(x)| \leq \frac{\eps}{2|\Om|},\;\; \forall x \in \mathcal{A}.
\end{gather*}	
So, for all $k \geq k_\eps$, we have,
	\begin{gather}\label{eq:a-1}
	\int_{\mathcal{A}}	|N_k(x)-N^*(x)|\,dx \leq \int_{\mathcal{A}}
	 \frac{\eps}{2|\Om|}\,dx \leq \frac{\eps}{2}.
	\end{gather}	
On another note, we observe that
\begin{equation*}
	\begin{aligned}
\int_{\Om \setminus \mathcal{A}}|N_k-N^*|\,dx&=\Iom |N_k-N^*|\chi_{\Om \setminus \mathcal{A}}\,dx\\
&\leq \Big(\Iom |N_k-N^*|\,dx  \Big)^{\frac{1}{R_{\al \ba}}}\big| \Om \setminus \mathcal{A} \big|^{\frac{1}{p_1^*}},
\end{aligned}
\end{equation*}	
Using the above $L^1(\Om)$ bound of $\{ N_k-N^*\}_{k \in \N}$ and $\big| \Om \setminus \mathcal{A} \big|^{\frac{1}{p_1^*}}<\frac{\eps}{2\mathcal{F}'}$,  we get
\begin{equation}\label{eq:a-2}
	\int_{\Om \setminus \mathcal{A}} |N_k-N^*|\,dx \leq \frac{\eps}{2}.
\end{equation}	
Combining (\ref{eq:a-1}) and (\ref{eq:a-2}) we see that,
\begin{equation*}
	\Iom |N_k-N^*|\,dx \leq \eps.
\end{equation*}	
Thus, $\{N_k\}_{k \in \N}$ converges to $N^*$ in $L^1(\Om)$. Consequently, without loss of generality, we get $\{N_k\}_{k \in \N}$ converges weakly to $N^*$ in $L^1(\Om)$ and similar arguments will lead to $\{M_k\}_{k \in \N}$ converges weakly to $M^*$ in $L^1(\Om)$. This finishes the proof. 
\end{proof}	

We now provide a direct consequence of Proposition \ref{prop-bounded1}.
{Firstly, observing that $\{M_k\}_{k \in \N}$ is bounded in $L^1(\Om)$, we
obtain the strong convergence of $[u_k]_{s_1, p_1}^{p_1}$ to $[u^*]^{p_1}_{s_1,p_1}$ in $L^{t_1}(\Om)$ for every $t_1<p_1$ and 
$[v_k]_{s_2, p_2}^{p_2}$ to $[v^*]^{p_2}_{s_2,p_2}$ in $L^{t_2}(\Om)$ for every $t_2<p_2$.} As a follow-up, we note that (see Lemma 3.3, \cite{Chen-16}),
\begin{equation}\label{eq:u_k}
    \frac{|u_k(x)-u_k(y)||^{p_1-2}(u_k(x)-u_k(y))}{|x-y|^{\frac{N+p_1s_1}{p'_1}}} \rightharpoonup \frac{|u^*(x)-u^*(y)||^{p_1-2}(u^*(x)-u^*(y))}{|x-y|^{\frac{N+p_1s_1}{p'_1}}}
\end{equation}
weakly in $L^{p'_1}(\mathbb R^{2N})$ and 
\begin{equation}\label{eq:v_k}
    \frac{|v_k(x)-v_k(y)||^{p_2-2}(v_k(x)-v_k(y))}{|x-y|^{\frac{N+p_2s_2}{p'_2}}} \rightharpoonup \frac{|v^*(x)-v^*(y)||^{p_2-2}(v^*(x)-v^*(y))}{|x-y|^{\frac{N+p_2s_2}{p'_2}}}
\end{equation}
weakly in $L^{p'_2}(\mathbb R^{2N})$. We now take care of an important corollary stated below.
\begin{corollary}
There holds,
\begin{equation*}
	\lim_{k \to \infty} \Iom M_k \phi_1 \,dx=\Iom M^* \phi_1 \,dx \, \text{ for all }\, \phi_1 \in X_{0, s_1, p_1}(\Om)
\end{equation*}	
and
\begin{equation*}
	\lim_{k \to \infty} \Iom N_k \phi_2 \,dx=\Iom N^* \phi_2 \,dx \, \text{ for all }\, \phi_2 \in X_{0, s_2, p_2}(\Om).
\end{equation*}	
\end{corollary}	
\begin{proof}
We notice that the proof holds true by Proposition \ref{prop-bounded1}
if $\phi_1$ and $\phi_2$ belong to $C_c^{\infty}(\Om)$. 
%and \textcolor{green}{$\Om$ having finite volume}. 
We now suppose that $\phi_1$ belongs to $X_{0, s_1, p_1}(\Om)$. Our objective is to show that 
\begin{gather}\lab{eq:M^*}
\Iom M_k \phi_1 \, dx \to \Iom M^* \phi_1\, dx.
\end{gather}
We choose $\eps>0$ sufficiently small. Since $C_c^{\infty}(\Om)$ is dense in $X_{0,s_1,p_1}(\Om)$, there exists $\Phi_\eps \in C_c^{\infty}(\Om)$ such that $[ \phi_1-\Phi_\eps]_{s_1, p_1} \leq \eps$.
Indeed, we have,
\begin{equation*}
\Iom M_k\phi_1 \,dx-\Iom M^*\phi_1\,dx=\Iom (M_k-M^*)(\phi_1-\Phi_\eps)\,dx+\Iom (M_k-M^*)\Phi_\eps \,dx,
\end{equation*}
which yields,
\begin{equation}\label{eq:+}
\Bigg|\Iom M_k\phi_1 \,dx-\Iom M^*\phi_1\,dx\Bigg|
\leq 
\Iom |M_k-M^*||\phi_1-\Phi_\eps|\,dx+ {\Big|
\Iom (M_k-M^*)\Phi_\eps \,dx\Big|.}
\end{equation}
Now, using H\"older's inequality, we have,
\begin{equation*}
    \Iom |M_k-M^*||\phi_1-\Phi_\eps|\,dx \leq \Big( \int_\Om|M_k-M^*|^{R_{\al,\ba}} \Big)^{\frac{1}{R_{\al,\ba}}}\Big(\int_\Om |\phi_1-\Phi_\eps|^{p_1^*}\,dx \Big)^{\frac{1}{p_1^*}}.
\end{equation*}
Applying the continuous embedding $X_{0, s_1, p_1}(\Om) \hookrightarrow L^{p_1^*}(\Om)$, we get, $| \phi_1-\Phi_\eps|_{L^{p_1^*}(\Om)} \leq \frac{\eps}{2}.$
{{
Using Proposition \ref{prop-bounded1} for $M_k$, for  $\eps>0$ sufficiently small and $k$ large enough, we deduce that
\begin{gather*}
    \Big|
\Iom (M_k-M^*)\Phi_\eps \,dx\Big| \leq \frac{\eps}{2}
\end{gather*}    
}}
Combining altogether we obtain using (\ref{eq:+}),
\begin{gather*}
\Big| \Iom M_k \phi_1 \, dx - \Iom M^* \phi_1\, dx \Big| \leq \frac{\eps}{2}+\frac{\eps}{2}=\eps,
\end{gather*}
thereby proving (\ref{eq:M^*}). In a similar fashion, we acquire 
\begin{gather*}
\Iom N_k \phi_2 \, dx \to \Iom N^* \phi_2\, dx,
\end{gather*}
for all $\phi_2 \in X_{0, s_2, q_2}(\Om)$. This wraps-up the proof.
\end{proof}	
The subsequent proposition is also thought-provoking for the readers and announced below:
\begin{proposition}\label{first-sol}
	The pair $(u^*, v^*)$ obtained in Lemma \ref{lem-7} is a solution of problem $(\mathcal{P}_{f_1,f_2})$.
\end{proposition}	
\begin{proof}
Let $\phi_i \in X_{0, s_i, p_i}(\Om)$ for $i=1,2$. We define $\mathcal{I}'|_{u}$ and $\mathcal{I}'|_{v}$ for $(u, v) \in X$ by
\begin{gather*}
    \< \mathcal{I}'|_{u}(u, v), \phi_1\>=\<\mathcal{I}'(u, v), (\phi_1,0)\>
\end{gather*}
and 
\begin{gather*}
    \< \mathcal{I}'|_{v}(u, v), \phi_2\>=\<\mathcal{I}'(u, v), (0, \phi_2)\>.
\end{gather*}
Putting $u=u_k$ and $v=v_k$ we get,
\begin{equation}\label{eq:-u}
    \begin{aligned}
    \<\mathcal{I}'|_{u}(u_k, v_k), \phi_1\>=\< (-\De)^{s_1}_{p_1} u_k, \phi_1\>-\Iom u_k |u_k|^{\al-1}|v_k|^{\ba+1}\phi_1\,dx-\<f_1, \phi_1\>-\<f_k, \phi_1\>
    \end{aligned}
\end{equation}
and 
\begin{equation}\label{eq:-v}
    \begin{aligned}
    \<\mathcal{I}'|_{v}(u_k, v_k), \phi_2\>=\< (-\De)^{s_2}_{p_2} v_k, \phi_2\>-\Iom |v_k|^{\al+1} v_k |v_k|^{\ba+1}\phi_2\,dx-\<f_2, \phi_2\>-\<g_k, \phi_2\>.
    \end{aligned}
\end{equation}
Note that 
$$
\frac{\phi_1(x)-\phi_1(y)}{|x-y|^{\frac{N+s_1p_1}{p_1}}} \in L^{p_1}(\Rn)\;\; \text{and}\;\; \frac{\phi_2(x)-\phi_2(y)}{|x-y|^{\frac{N+s_2p_2}{p_2}}} \in L^{p_2}(\Rn)
$$
which gives 
$$
\lim_{k\to \infty}\< (-\De)^{s_1}_{p_1} u_k, \phi_1\> = \< (-\De)^{s_1}_{p_1} u^*, \phi_1\> \;\;\text{and}\;\; \lim_{k\to \infty}\< (-\De)^{s_2}_{p_2} v_k, \phi_2\> = \< (-\De)^{s_2}_{p_2} v^*, \phi_2\>.
$$
On taking the limit and using (\ref{eq:u_k}) and (\ref{eq:v_k})
we obtain,
\begin{equation*}
    \begin{aligned}
    \<\mathcal{I}'|_{u}(u_k, v_k), \phi_1\>=\< (-\De)^{s_1}_{p_1} u^*, \phi_1\>-\Iom u^* |u^*|^{\al-1}|v^*|^{\ba+1}\phi_1\,dx-\<f_1, \phi_1\>
    \end{aligned}
\end{equation*}
and
\begin{equation*}
    \begin{aligned}
    \<\mathcal{I}'|_{v}(u_k, v_k), \phi_2\>=\< (-\De)^{s_2}_{p_2} v^*, \phi_2\>-\Iom  |v^*|^{\al+1}|v^*|^{\ba+1} v^* \phi_2\,dx-\<f_2, \phi_2\>.
    \end{aligned}
\end{equation*}
Moreover, using (\ref{eq:I}) and (\ref{eq:I'}) we obtain 
\begin{equation*}
    \begin{aligned}
    \< (-\De)^{s_1}_{p_1} u^*, \phi_1\>-\Iom u^* |u^*|^{\al-1}|v^*|^{\ba+1}\phi_1\,dx-\<f_1, \phi_1\>=0
    \end{aligned}
\end{equation*}
for every $\phi_1$ in $X_{0, s_1, p_1}(\Om)$ and 
\begin{equation*}
    \begin{aligned}
    \< (-\De)^{s_2}_{p_2} v^*, \phi_2\>-\Iom  |v^*|^{\al+1}|v^*|^{\ba+1}v^* \phi_2 \,dx-\<f_2, \phi_2\>=0
    \end{aligned}
\end{equation*}
for every $\phi_2$ in $X_{0, s_2, p_2}(\Om)$
As a result, we deduce that,
\begin{equation*}
    \begin{aligned}
    \begin{cases}
    (-\De)^{s_1}_{p_1}u^*&=u^*|u^*|^{\al-1}|v^*|^{\ba+1}+f_1 \,\text{ in }\, \Om,\\
    (-\De)^{s_2}_{p_2}v^*&=u^*|u^*|^{\al+1}|v^*|^{\ba-1}+f_2 \,\text{ in }\, \Om.
    \end{cases}
    \end{aligned}
\end{equation*}

This implies, $(u^*, v^*)$ is a weak solution of $(\mathcal{P}_{f_1, f_2})$.
Contrarily, this solution $(u^*, v^*)$ satisfies the following:
\begin{itemize}
    \item [(i)]
    $\< \mathcal{I}'(u^*, v^*), (u^*, v^*)\>=0$,
    \item [(ii)]
    $\mathcal{I}(u^*, v^*)=m<0$.
\end{itemize}
By (i), we note that $(u^*, v^*)$ belongs to $\Theta$. As $(u^*, v^*)$ is the solution of $(\mathcal{P}_{f_1, f_2})$, we obtain (i) by taking into account $(\phi_1, \phi_2)=(u^*, v^*)$.

To see the proof of (ii), we note that (i) implies 
$m \leq \mathcal{I}(u^*, v^*)$, since $m:=\inf_{(u,v) \in \Theta}\mathcal{I}(u,v)$. Otherwise, due to the weak semi-continuity of $\mathcal{I}|_{\Theta}$ and $\mathcal{I}(u_k, v_k)< m + \frac{1}{k}$ yields,
$$ \mathcal{I}(u^*, v^*) \leq \liminf_{k \to +\infty} \mathcal{I}(u_k, v_k) \leq m.
$$
This in turn yields us,
\begin{gather*}
    m=\lim_{k \to \infty} \mathcal{I}(u_k, v_k)=\mathcal{I}(u^*, v^*).
\end{gather*}
Using Lemma \ref{lem-4}, we get $\mathcal{I}(u^*, v^*)<0$. This finishes our proof.
\end{proof}

\begin{remark}\label{remark}
Before wrapping-up the article, as a future direction it would be inspiring to study analogous set of problems by considering even more general kernel, again both fractional and nonlinear, by taking into account rough coefficients. Thus, by replacing the fractional $(p_1, p_2)$ Laplacians in ${(\mathcal{P}_{f_1,f_2})}$ with the nonlinear integro-differential operator given by
$$\mathcal{L}_K u(x)=P.V. \int_{\Rn}|u(y)-u(x)|^{p-2}(u(y)-u(x))K(x,y)dy,\, x \in \Rn$$
where the symmetric function $K$ is a Gagliardo-type kernel (namely, an $(s, p)$-kernel) with measurable coefficients. Owing to this aspiration, we cite (\cite{Ref-1},\cite{Ref-2}) for interested readers which treats with the general nonlinear fractional operators above.
%could be the starting point in order to study such a generalization.
\end{remark}
%}}

\end{document}